\documentclass[11pt,a4paper]{article}
\usepackage[margin=1.15in]{geometry}
\usepackage{amsmath,amssymb,amsthm}
\usepackage{tikz}
\usepackage{fancyhdr}

\fancypagestyle{plain}{\fancyhf{}}

\theoremstyle{plain}
\newtheorem{theorem}{Theorem}[section]
\newtheorem{corollary}[theorem]{Corollary}
\newtheorem{proposition}[theorem]{Proposition}
\newtheorem{problem}{Problem}
\theoremstyle{definition}
\newtheorem{definition}[theorem]{Definition}
\newtheorem{example}[theorem]{Example}
\theoremstyle{remark}
\newtheorem{counterexample}{Counterexample}

\newcommand{\Spec}{\mathrm{Spec}_P}
\newcommand{\PE}{PE}

\begin{document}
	
	% ---------- title ----------
	\begin{center}
		{\Large\scshape Bounds on Path Energy of Graphs}\\[1.2em]
		{\scshape Amol Narke$^{1}$, Prashant Malavadkar$^{2}$, Maruti Shikare$^{3}$}\\[0.4em]
		{\small $^{1,2}$Department of Mathematics and Statistics, Dr.\ Vishwanath Karad MIT-World Peace University, Pune-38, India.}\\[0.2em]
		{\small $^{3}$Department of Mathematics, School of Computational Sciences,
		JSPM University, Pune, Maharashtra, India.}\\[0.2em]
		{\small e-mail: \textit{1.\ amol.narke@mitwpu.edu.in \ 2.\ prashant.malavadkar@mitwpu.edu.in \ 3.\ mmshikare@gmail.com}}
	\end{center}
	
	\vspace{0.6em}
	\noindent\rule{\textwidth}{0.6pt}
	\begin{center}
		{\small\bfseries Abstract}
	\end{center}
	\begin{center}
		\parbox{0.88\textwidth}{\small
			The \emph{path matrix} $P(M)$ of a graph $M$ has as its $(i,j)$ entry the maximum number of internally vertex-disjoint paths joining the vertices $v_i$ and $v_j$; its eigenvalues are the \emph{path eigenvalues} and their absolute values sum to the \emph{path energy} $\PE(M)$. We have discovered new upper constraints on path energy in this study, expressed in terms of a graph's maximum degree. Additionally, a relationship between a graph's energy and path energy is given.}
	\end{center}
	\noindent\rule{\textwidth}{0.6pt}
	
	% =====================================================
	\section{Introduction}
	
	The energy of a graph was introduce in 1978 by I.\ Gutman [6]. If $M$ be a simple graph then $P(M)$ be its path matrix defined by Patekar and Shikare in 2016 as follows,
	
	\begin{definition}[9]
		Let $M$ be a simple graph, and let $V(M)=v_1,v_2,\dots,v_k$. be its vertex set. Define the path matrix of graph $M$ as matrix $P=(p_{ij})$ of dimension $k\times k$, where $p_{ij}=0$ if $i=j$ and equals the maximum number of vertex disjoint paths from $v_i$ to $v_j$ if $i\neq j$.
	\end{definition}
	
	The spectrum of path matrix is denoted by $\Spec(M)$. Consider the following example,
	
	\begin{center}
		\begin{tikzpicture}[scale=1.0,every node/.style={font=\small}]
			\coordinate (v1) at (0,1);
			\coordinate (v2) at (3,1);
			\coordinate (v4) at (0,0);
			\coordinate (v3) at (3,0);
			\draw (v1)--(v2) node[midway,above]{p};
			\draw (v2)--(v3) node[midway,right]{q};
			\draw (v4)--(v3) node[midway,below]{r};
			\draw (v4)--(v2) node[midway,above,xshift=-4pt]{s};
			\foreach \v in {v1,v2,v3,v4} \fill (\v) circle (3.5pt);
			\node[above left] at (v1) {1};
			\node[above right] at (v2) {2};
			\node[below right] at (v3) {3};
			\node[below left] at (v4) {4};
		\end{tikzpicture}\\[2pt]
		Figure 1: $\mathbf{M}$
	\end{center}
	
	The path matrix for a graph $M$ is,
	\[
	P(M)=\begin{bmatrix}0&1&1&1\\1&0&2&2\\1&2&0&2\\1&2&2&0\end{bmatrix}
	\]
	$\Spec(M)=\{-2,-2,4.6457,-0.6457\}$.
	
	\begin{definition}[12]
		The path energy of graph $M$ is represented by $PE=PE(M)$, which is equal to sum of absolute values of the path eigenvalues.
	\end{definition}
	
	The path energy $PE(M)$ of the graph $M$ shown in Figure 1 is, $|-2|+|-2|+|4.6457|+|-0.6457|=9.2914$.
	
	The path eigenvalues of the $r$-regular $r$-connecetd graph $M$ on $k$ vertices are $-r$ with multiplicity $k-1$ and $r(k-1)$ with multiplicity $1$[12]. 
	In 2016, Patekar and Shikare [9] have found path eigenvalues of some graphs. In Table 1 we have observed that all graphs in Table 1 has only one positive eigenvalue which is path spectral radius ($\rho$) of that graph. Thus for such graphs we can easily obtained path energy which is equal to $PE(M)=2\rho$
	
	\begin{proposition}
		Let $M$ be a simple connected graph with only one positive path eigenvalue then this path eigenvalue is path spectral radius of $M$ called $\rho$. Moreover $PE(M)=2\rho$.
	\end{proposition}
	\begin{proof}
		Let $P(M)$ be a path matrix of simple connected graph $M$ on $p$ vertices with path eigenvalues $\beta_1,\beta_2,\dots,\beta_n$. Without loss of generality assume that $\beta_1>0$ and for $i\ge2$, $\beta_i<0$. As trace of the path matrix is always zero this implies that $\beta_1+\beta_2+\beta_3+\dots+\beta_n=0$ gives us, $\beta_1=-\beta_2-\beta_3-\dots-\beta_n$. On the contrary assume that $\beta_1$ is not a path spectral radius of $M$ then for some integer $k$ such that $2\le k\le p$, $|\beta_k|$ be a path spectral radius of $M$ that is $|\beta_k|>|\beta_1|$. Consider, $\beta_1=-\beta_2-\beta_3-\dots-\beta_n$ implies that $|\beta_1|=|\beta_2|+|\beta_3|+\dots+|\beta_n|$. This contradict the fact that $|\beta_k|>|\beta_1|$. Thus $\beta_1$ is path spectral radius of $M$. Moreover, $PE(M)=|\beta_1|+|\beta_2|+|\beta_3|+\dots+|\beta_n|$ as $\beta_1=-(\beta_2+\beta_3+\dots+\beta_n$). Gives us $PE(M)=2\beta_1=2\rho$.
	\end{proof}
	
	Following is the list of some graphs having only one positive path eigenvalue ($\rho$) and path energy $2\rho$.
	
	\begin{table}[h]
		\centering
		\renewcommand{\arraystretch}{1.5}
		\begin{tabular}{|c|c|}
			\hline
			Graphs & Path eigenvalues\\ \hline
			$K_p$ & $\begin{pmatrix}(p-1)^2 & (-p+1)\\ 1 & (p-1)\end{pmatrix}$\\ \hline
			($r$-regular graph)$_p$ & $\begin{pmatrix}r(p-1) & (-r)\\ 1 & (p-1)\end{pmatrix}$\\ \hline
			$T_p$ & $\begin{pmatrix}(p-1) & (-1)\\ 1 & (p-1)\end{pmatrix}$\\ \hline
			$\mathcal{Q}_p$ & $\begin{pmatrix}p(2^p-1) & (-p)\\ 1 & (2^p-1)\end{pmatrix}$\\ \hline
			$\mathcal{Q}_p\square\mathcal{Q}_q$ & $\begin{pmatrix}(q+p)(2^{q+p}-1) & -(q+p)\\ 1 & 2^{q+p}-1\end{pmatrix}$\\ \hline
			$W_p$ & $\begin{pmatrix}3(p-1) & -3\\ 1 & (p-1)\end{pmatrix}$\\ \hline
		\end{tabular}
		\\[4pt] Table 1: Graphs having only one positive path eigenvalue
	\end{table}
	
	% ---------- Problem 1 ----------
	\begin{problem}\label{prob:1}
		Does the path matrix of every $2$-connected graph $M$ have exactly one positive eigenvalue?
	\end{problem}
	
	\begin{counterexample}[to Problem~\ref{prob:1}]\label{ce:1}
		The answer is negative. Let $M=K_2\vee\overline{K_5}$ ($=K_{1,1,5}$): its vertex set is $\{u,v,w_1,\dots,w_5\}$ and its edges are $uv$ and $uw_i$, $vw_i$ for $i=1,\dots,5$ ($11$ edges, $7$ vertices). Deleting any single vertex leaves a connected graph, so $M$ is $2$-connected. The maximum numbers of vertex disjoint paths are $p_{uv}=6$ (the edge $uv$ and the five paths $u\,w_i\,v$), while $p_{uw_i}=p_{vw_i}=2$ and $p_{w_iw_j}=2$ (the paths $w_i u w_j$ and $w_i v w_j$; $\deg w_i=2$). Hence, with the ordering $u,v,w_1,\dots,w_5$,
		\[
		P(M)=\begin{bmatrix}
			0&6&2&2&2&2&2\\
			6&0&2&2&2&2&2\\
			2&2&0&2&2&2&2\\
			2&2&2&0&2&2&2\\
			2&2&2&2&0&2&2\\
			2&2&2&2&2&0&2\\
			2&2&2&2&2&2&0
		\end{bmatrix}.
		\]
		The vector $e_u-e_v$ is an eigenvector for the eigenvalue $-6$; every vector supported on the $w_i$'s with coordinate sum $0$ is an eigenvector for $-2$ (multiplicity $4$). On vectors that are constant on $\{u,v\}$ (value $a$) and constant on $\{w_i\}$ (value $b$), $P$ acts as
		\[
		\begin{pmatrix}6&10\\4&8\end{pmatrix},
		\]
		whose eigenvalues are the roots of $\lambda^2-14\lambda+8=0$, i.e.\ $7\pm\sqrt{41}$. Therefore
		\[
		\Spec(M)=\{\,7+\sqrt{41}\approx13.4031,\;7-\sqrt{41}\approx0.5969,\;-6,\;-2,-2,-2,-2\,\},
		\]
		so $P(M)$ has \emph{two} positive eigenvalues although $M$ is $2$-connected. (Here $PE(M)=28$.)
		
		More generally, for $K_2\vee\overline{K_n}$ the symmetric part is $\begin{pmatrix}n+1&2n\\4&2(n-1)\end{pmatrix}$ with trace $>0$ and determinant $2(n^2-4n-1)>0$ for every $n\ge5$, so every $K_2\vee\overline{K_n}$ with $n\ge5$ is a $2$-connected graph with two positive path eigenvalues. Another example is two copies of $K_4$ joined by two disjoint edges (spectrum has positive eigenvalues $\approx17.5125$ and $\approx1.5616$).
	\end{counterexample}
	
	Note that: The converse of the statement in Problem 1 is not true that is, If path matrix has exactly one positive eigenvalue then its corresponding graph may not be $2$-connected.
	
	\begin{example}
		The eigenvalues of any tree on $p$ vertices have path spectrum
		
		 $\{(p-1),-1,-1,\dots,-1_{(p-1)times.}\}$
	\end{example}
	
	\begin{example}
		In [8], Ilic proved that path matrix of unicyclic graph on $p$ verices with cycle length $k$ has two positive path eigenvalues if $p\ge8$ and $3\le k\le(p-3)$.
	\end{example}
	
	% ---------- Problem 2 ----------
	\begin{problem}\label{prob:2}
		If the graph $M$ can be decomposed in to $k$ $2$-connected components, does the path matrix of $M$ have exactly $k$ positive eigenvalues?
	\end{problem}
	
	\begin{counterexample}[to Problem~\ref{prob:2}]\label{ce:2}
		The answer is again negative, and fails in both directions.
		\begin{enumerate}
			\item[(a)] \emph{Fewer positive eigenvalues than components.} Let $M$ be the bowtie (butterfly) graph: two triangles $\{c,a_1,a_2\}$ and $\{c,b_1,b_2\}$ sharing the vertex $c$. It decomposes into $k=2$ $2$-connected components. Here $p_{ca_i}=p_{cb_i}=2$, $p_{a_1a_2}=p_{b_1b_2}=2$ and $p_{a_ib_j}=1$. The vectors $e_{a_1}-e_{a_2}$ and $e_{b_1}-e_{b_2}$ give eigenvalue $-2$ twice, $e_{a_1}+e_{a_2}-e_{b_1}-e_{b_2}$ gives eigenvalue $0$, and on vectors of the form $(x;y,y,y,y)$ (value $x$ at $c$) $P$ acts as $\begin{pmatrix}0&8\\2&4\end{pmatrix}$ with eigenvalues $2\pm2\sqrt5$. Thus
			\[
			\Spec(M)=\{\,2+2\sqrt5,\;0,\;-2,\;-2,\;2-2\sqrt5\,\},
			\]
			and $P(M)$ has exactly \emph{one} positive eigenvalue, not two.
			\item[(b)] \emph{More positive eigenvalues than components.} The graph $K_2\vee\overline{K_5}$ of Counterexample~\ref{ce:1} is itself $2$-connected ($k=1$) but has two positive path eigenvalues. (Problem~\ref{prob:2} with $k=1$ is exactly Problem~\ref{prob:1}.)
			\item[(c)] If single edges ($K_2$) are counted as components, every tree on $p\ge3$ vertices has $k=p-1$ components but, by Example 1.8, exactly one positive path eigenvalue.
		\end{enumerate}
	\end{counterexample}
	
	% =====================================================
	\section{New bounds for path energy}
	
	In the following proposition we have found upper bound for row sum of path matrix in term of degree of vertices.
	
\begin{proposition}
	Let $M$ be a simple graph of order $p$. Then, for every
	$i\in\{1,\ldots,p\}$,
	\[
	\sum_{\substack{j=1\\j\ne i}}^{p}p_{ij}
	\leq(p-1)d(v_i).
	\]
\end{proposition}
\begin{proof}
	Fix $i$. For every $j\ne i$, each internally vertex-disjoint
	$v_iv_j$-path uses a distinct edge incident with $v_i$. Therefore
	$p_{ij}\leq d(v_i)$. Summing over all $j\ne i$ gives
	\[
	\sum_{\substack{j=1\\j\ne i}}^{p}p_{ij}
	\leq\sum_{\substack{j=1\\j\ne i}}^{p}d(v_i)
	=(p-1)d(v_i),
	\]
	as required.
\end{proof}
	
	Following theorem gives the upper bound for absolute eigenvalue of path matrix in terms of maximum vertex degree of a graph.
	
	\begin{theorem}
		Let $M$ be a graph on $p$ vertices, with path matrix $P(M)$. $\Delta$ be the maximum degree of $M$. If $\beta$ is an eigenvalue of $P(M)$ then $|\beta|\le(p-1)\Delta$.
	\end{theorem}
	\begin{proof}
		Let $x$ be an eigenvector of $P(M)$ corresponding to eigenvalue $\beta$, then $Px=\beta x$. Write $x=(x_1,x_2,x_3,\dots,x_p)^t$ where without loss of generality, $|x_1|=\max_{1\le i\le p}|x_i|$.
		\[
		|\beta||x_1|=\sum_{j\neq i=1}^{p}p_{1,j}x_j\le|x_1|\sum_{j\neq i=1}^{p}p_{1,j}\le|x_1|(p-1)\deg(v_i)\le|x_1|(p-1)\Delta\le|x_1|(p-1)\Delta.
		\]
	\end{proof}
	
	Aleksandar Ilic et al.\ [8] proved the conjecture $PE(M)\le2(p-1)^2$ stated in Shikare et.\ al [9]. In the following corollary we obtained a bound for path energy of graph in term of total number of edges.
	
	\begin{corollary}
		Let $M$ be graph on $p$ vertices and $m$ edges. Then $PE(G)\le2(p-1)m$.
	\end{corollary}
	\begin{proof}
		Theorem 2.2 gives us
		\begin{equation}
			|\beta|\le(p-1)\deg(v_i)
		\end{equation}
		This is true for every $|\beta|$, taking summation over $i$ we get,
		\begin{equation}
			\sum_{i=1}^{p}|\beta|\le(p-1)\sum_{i=1}^{p}\deg(v_i)\le(p-1)(2m).
		\end{equation}
		Thus we get $PE(M)\le2(p-1)m$.
	\end{proof}
	
	\begin{corollary}
		Let $M$ be a graph on $p$ vertices. $\Delta$ be the maximum vertex degree of $M$ then $PE(M)\le p(p-1)\Delta$.
	\end{corollary}
	\begin{proof}
		We have $|\beta|\le(p-1)\Delta$ for every $|\beta|$. Taking summation over $\beta\in\Spec$, we get $PE(M)\le p(p-1)\Delta$.
	\end{proof}
	
	A relation between energy of graph and path energy of graph is given in the following corollary.
	
	\begin{corollary}
		Let $M$ be a connected graph on $p$ vertices and $\Delta$ be the maximum vertex degree of a graph $M$ then $E(M)\le\dfrac{p}{2}PE(M)\le\dfrac{p^2(p-1)\Delta}{2}$. Equality holds if $M=K_2$.
	\end{corollary}
	\begin{proof}
		Let $M$ be a connected graph on $p$ vertices. Ilic et al.\ [8] proved the inequality $2(p-1)\le PE(M)$. By using Corollary 2.4, $PE(M)\le p(p-1)\Delta$. Thus,
		\begin{equation}
			(p-1)\le\frac{PE(M)}{2}\le\frac{p(p-1)\Delta}{2}.
		\end{equation}
		Let $A(M)$ be an adjacency matrix of $M$. And $x$ be an eigenvector of $A(M)$ corresponding to eigenvalue $\beta'$ then,
		\[
		Ax'=\beta'x'
		\]
		Write $x'=(x'_1,x'_2,x'_3,\dots,x'_p)^t$ where without loss of generality, $|x'_1|=\max_{1\le i\le p}|x'_i|$. Now,
		\begin{equation}
			|\beta'||x'_i|=\sum_{j=1}^{p}a_{1j}x'_j\le|x'_1|\sum_{j=1}^{p}a_{1j}\le|x'_1|\deg(v_i)\le|x'_1|\Delta(v)\le|x'_1|\Delta(v).
		\end{equation}
		Taking summation over $\beta'\in\mathrm{Spec}_M$ we get, $E(M)\le p\Delta$. As $M$ is simple connected graph, $\Delta=(p-1)$. Thus,
		\begin{equation}
			E(G)\le p(p-1).
		\end{equation}
		Combining equations 3 and 5 we get, $E(M)\le\dfrac{p}{2}PE(M)\le\dfrac{p^2(p-1)\Delta}{2}$. Further the equality holds if $M$ is $K_2$.
	\end{proof}

\end{document}